\documentclass[11pt]{article}
\usepackage{bbm,dsfont,yfonts}
\usepackage[all,cmtip]{xy}
\usepackage{amsfonts,amssymb,amsmath,amscd,amsthm,mathtools}
\usepackage{mathrsfs}
\usepackage{latexsym}
\usepackage{color}
\usepackage{hyperref}
\usepackage{verbatim,enumitem}
\usepackage{tikz-cd}
\usepackage{booktabs} % Allows the use of \toprule, \midrule and \bottomrule in tables
\input diagxy
\usepackage[total={160mm,230mm}]{geometry}

\DeclareMathOperator{\with}{\&}
\DeclareMathOperator{\ldd}{/}
\DeclareMathOperator{\rdd}{\backslash}

\theoremstyle{plain}
  \newtheorem{thm}{Theorem}[section]
  \newtheorem{lem}[thm]{Lemma}
  \newtheorem{prop}[thm]{Proposition}
  
\theoremstyle{definition}
  \newtheorem{defn}[thm]{Definition}
  
  \newtheorem{exmp}[thm]{Example}
  \newtheorem{rem}[thm]{Remark}

\newtheorem*{SA}{Standing Assumption}

\newcommand{\ra}{\rightarrow}

\newcommand{\lam}{\lambda}

\newcommand{\sub}{{\rm sub}}
\newcommand{\id}{{\rm id}}

\newcommand{\sfe}{{\sf e}}
\newcommand{\sfm}{{\sf m}}

\newcommand{\sQ}{{\sf Q}}
\newcommand{\setQ}{{\sf Set}(\sQ)}
\newcommand{\Set}{\sf Set}
\newcommand{\bv}{\bigvee}
\newcommand{\bw}{\bigwedge}

\newcommand{\FQ}{\sQ\text{-}{\sf Fil}}

\allowdisplaybreaks

\begin{document}
\title{The double contravariant powerset monad in the Goguen category of fuzzy sets}%\thanks{This work is supported by the National Natural Science Foundation of China  (No. 11871358)}}

\author{Sijia Lu, Dexue Zhang  \\ {\small School of Mathematics, Sichuan University, Chengdu, China}\\  {\small sijialu1027@qq.com, dxzhang@scu.edu.cn}  }
\date{}
\maketitle

\begin{abstract}A monad is constructed in the Goguen category of fuzzy sets valued in a unital quantale, which is an analog  of the double contravariant powerset monad in the category of sets. With help of this monad it is proved that the Goguen category of fuzzy sets is dually monadic over itself.
\vskip 1pt

\end{abstract}

\section{Introduction}
In order to construct a foundation for fuzzy set theory, Goguen \cite{Goguen69b,Goguen74} introduced a category ${\sf Set}(L)$, now called the Goguen category of $L$-fuzzy sets (or  $L$-sets),  where $L$ is a complete lattice, often endowed with some extra structures. Goguen has obtained a characterization of such categories by a relatively simple system of axioms.

Let $L$ be a complete lattice.   Then an $L$-set is a pair $(X,\alpha)$, where $X$ is a set and $\alpha\colon X\to L$ is a function. The set $X$ is the carrier of the $L$-set, the complete lattice  $L$ is the truth value set, and the value $\alpha(x)$ is the membership degree of the point $x$ in the $L$-set.  A morphism from $(X,\alpha)$ to $(Y,\beta)$ is a function $f\colon X\to Y$ such that $\alpha\leq \beta\circ f$, in which case we say that $f\colon(X,\alpha)\to(Y,\beta)$ is a Goguen map or satisfies the Goguen condition. $L$-sets and their morphisms constitute a category ${\sf Set}(L)$ --- the Goguen category of $L$-sets.
The  category ${\sf Set}(L)$ provides a nice framework for the study and application of fuzzy sets, it enjoys many pleasant categorical properties, particularly when  $L$ possesses rich structures, as demonstrated in Goguen \cite{Goguen69,Goguen74}, Pultr \cite{Pultr76a,Pultr76b}, H\"{o}hle and Stout \cite{HS1991}, and  Stout \cite{Stout92}.

The \emph{double contravariant powerset monad} in the category of sets (see e.g. \cite[Example 2.11]{Manes2003}) and its submonads, the filter monad and the ultrafilter monad in particular, are among the fundamental constructions of sets and play  important roles  in category theory, algebra, topology, and other disciplines, see e.g. \cite{Monoidal top,Manes76,Manes2003,Manes2010}. Since the Goguen category ${\sf Set}(L)$ is not a topos unless $L$ is a singleton set, its behavior is quite different from the category  of sets when ``powerobjects'' are concerned \cite{Stout92}. It is natural to ask, though not a topos, whether there exist analogous constructions in the category ${\sf Set}(L)$.
In this paper, in the circumstance that the truth value set is a unital quantale $\sQ$,  a monad $(\mathfrak{P},\mu,\eta)$ is constructed in the category $\setQ$ of ($\sQ$-valued) fuzzy sets. This monad is an analog of the double contravariant powerset monad of sets, and it is a lifting of the double contravariant $\sQ$-powerset monad   in \cite[Remark 1.2.7]{Hoehle2001}. Some basic properties of the monad $(\mathfrak{P},\mu,\eta)$ are examined. The main results include: (i) The category of the Eilenberg-Moore algebras of $(\mathfrak{P},\mu,\eta)$ is equivalent to the opposite category of $\setQ$, hence the Goguen category of fuzzy sets is dually monadic over itself,   adding another one to the list of pleasant properties of the Goguen category.   (ii) For a commutative quantale, the (covariant) powerset monad in $\setQ$ constructed in \cite{EKS12} is a submonad of the double contravariant $\sQ$-powerset monad.

\section{Preliminaries}
For category theory we refer to   Mac{\thinspace}Lane \cite{MacLane1998} or  Riehl \cite{Riehl}; for monads in the category of sets we refer to Manes \cite{Manes2003}; for quantale theory we refer to  Rosenthal \cite{Rosenthal1990}. In this preliminary section, we just recall some basic ideas about quantales and monads,  the aim is to fix notations.

\subsection*{Quantales}
A \emph{unital quantale} \cite{Rosenthal1990} (also known as a \emph{complete residuated lattice}  \cite[page 178]{Galatos2007})
\[\sQ=(\sQ,\with,k)\]
is a monoid with $k$ being the unit, such that the underlying set $\sQ$ is a complete lattice (with a top element $1$ and a bottom element $0$) and the multiplication $\with$ distributes over arbitrary suprema, i.e.,
\[p\with\Big(\bv_{i\in I}q_i\Big)=\bv_{i\in I}p\with q_i\quad\text{and}\quad\Big(\bv_{i\in I}p_i\Big)\with q=\bv_{i\in I}p_i\with q\]
for all $p,q,p_i,q_i\in\sQ$ $(i\in I)$.

Quantales abound in mathematics; numerous examples are presented in \cite{EGHK,Rosenthal1990}. As advocated in Goguen \cite{Goguen67,Goguen69,Goguen74} (where quantales are called complete lattice ordered semigroup), quantales are natural candidates for truth-value tables in the theory of fuzzy sets.

For each $q\in \sQ$, the map $-\with q\colon\sQ\to \sQ$ has a right adjoint \[-\ldd q\colon \sQ\to\sQ, \quad r\ldd q=\bv\{p\in\sQ\mid p\with q\leq r\},\] called the \emph{left  implication} of $\with$.

For each $p\in \sQ$, the map $p\with-\colon\sQ\to \sQ$ has a right adjoint \[p\rdd -\colon \sQ\to\sQ, \quad p\rdd r=\bv\{q\in\sQ\mid p\with q\leq r\},\] called the \emph{right  implication} of $\with$.

The left and the right implications satisfy that
\[p\leq r\ldd q\iff p\with q\leq r\iff   q\leq p\rdd r\]
for all $p,q,r\in\sQ$.

A quantale $\sQ$ is \emph{commutative}  if $p\with q=q\with p$ for all $p,q\in\sQ$, in which case we write     \[p\ra q:=q\ldd p=p\rdd q\]     for all $p,q\in\sQ$.

Some basic properties of the left and the right implications are listed below.

\begin{prop}{\rm(\cite{Galatos2007,Rosenthal1990})} \label{prop of impli} Let $(\sQ,\with,k)$ be a unital quantale. \begin{enumerate}[label={\rm(\roman*)}] \setlength{\itemsep}{0pt}
\item $k\leq y\ldd x\iff x\leq y\iff k\leq x\rdd y$.
\item $x\ldd k=x=k\rdd x$.

\item $(y\rdd z)\ldd x=y\rdd (z\ldd x)$.
\item $(y\ldd  x)\with x\leq y, \ x\with (x\rdd  y)\leq y$.
\item $(z\ldd y)\ldd x=z\ldd (x\with y)$,
$x\rdd (y\rdd z)=(y\with x)\rdd z$.
\item $(z\ldd y)\with (y\ldd x)\leq z\ldd x$, $(x\rdd y)\with (y\rdd z)\leq x\rdd z$.
\end{enumerate}\end{prop}

\begin{SA}Throughout this paper, $\sQ=(\sQ,\with,k)$ always denotes a unital quantale with at least two elements, unless otherwise specified.\end{SA}

%The following notations  are  needed in this paper.

For each $r\in\sQ$ and each element $x$ of a set $X$, we write $r_x$ for the element of $\sQ^X$   given by $$r_x(y)=\begin{cases}r & y=x,\\ 0 & y\not=x.\end{cases}$$

Let $X$ be  a set. For all $\lambda,\gamma\in\sQ^X$, we write $\lambda\swarrow\gamma$ and $\gamma\searrow\lambda$ for elements of $\sQ$ given  by \[\lambda\swarrow\gamma =\bw_{x\in X}\lambda(x)\ldd\gamma(x)\] and \[ \gamma\searrow\lambda=\bw_{x\in X}\gamma(x)\rdd\lambda(x).\]

It is easily verified that \begin{itemize} \item $(\bw_i\lambda_i)\swarrow\gamma =\bw_i( \lambda_i \swarrow\gamma)$; ~ $ \lambda\swarrow(\bv_i\gamma_i) =\bw_i( \lambda \swarrow\gamma_i)$. \item $\gamma\searrow (\bw_i\lambda_i)= \bw_i(\gamma\searrow  \lambda_i)$; ~ $(\bv_i\gamma_i)\searrow\lambda =\bw_i(\gamma_i \searrow\lambda)$.\end{itemize}

Let $X,Y$ be sets and $f\colon X\to Y$ be a map. For each $\gamma\in\sQ^X$,  we write $f(\gamma)$ for the image  of $\gamma$; that is, \[f(\gamma)(y)=\bv\{\gamma(x)\mid f(x)=y\} \] for all $y\in Y$. For each $\lambda\in\sQ^Y$, we write $f^{-1}(\lambda)$ for the inverse image  of $\lambda$; that is,  $f^{-1}(\lambda)=\lambda\circ f$.

\begin{lem} \label{image vs preimage} Let $f\colon X\to Y$ be a map. Then for all $\alpha,\gamma\in\sQ^X$ and $\beta\in\sQ^Y$, \begin{enumerate}[label={\rm(\roman*)}] \setlength{\itemsep}{0pt} \item $\alpha\searrow\gamma\leq f(\alpha)\searrow f(\gamma)$,~ $\gamma\swarrow\alpha\leq f(\gamma)\swarrow f(\alpha)$; \item  $f(\alpha)\searrow\beta =\alpha\searrow f^{-1}(\beta)$,~ $\beta\swarrow f(\alpha)= f^{-1}(\beta)\swarrow\alpha$; \item  $f(\alpha)\leq\beta\iff \alpha\leq\beta\circ f$. \end{enumerate} \end{lem}

For each unital quantale $\sQ$, the Goguen category $\setQ$  refers to the category for which  \begin{itemize} \item  an object is a fuzzy set  $(X,\alpha)$; \item  a morphism (called a Goguen map)  $f$ from $(X,\alpha)$ to $(Y,\beta)$ is a map $f\colon X\to Y$ such that $\alpha\leq\beta\circ f$;  \item  composition is the usual composition of maps. \end{itemize}

\subsection*{Monads and their algebras}
A  monad  $\mathbb{T}=(T,\mu,\eta)$ in a category $\mathcal{A}$ consists of a functor $T\colon\mathcal{A}\to\mathcal{A}$ and two natural transformations \[\mu\colon T^2\to T, \quad \eta\colon{\rm id}_\mathcal{A}\to T\] which make the following  diagrams commutative:
$$\bfig \square[T^3`T^2`T^2`T;T\mu`\mu T`\mu`\mu]
\Vtrianglepair(1100,0)/>`<-`@{=}`>`@{=}/[T`T^2`T`T;\eta T`T\eta`
`\mu` ] \efig$$ The natural transformations $\eta$ and $\mu$ are called the  unit and  the  multiplication  of the monad, respectively.

\begin{exmp}(\cite[Example 2.16]{Manes2003})\label{powerset monad} The  (covariant) powerset functor  on {\sf Set} is the functor \[\exp \colon {\sf Set}\to{\sf Set}\] that assigns to each   $f\colon X\to Y$ the map $$\exp f\colon 2^X\to 2^Y, \quad A\mapsto f(A).$$ The functor $\exp$ gives rise to a monad \[ (\exp,\sfm,\sfe)\] in $\Set$,  where for each set $X$,  \begin{itemize}  \item  $\sfe_X\colon X\to 2^X$ sends each $x\in X$ to the singleton set $\{x\}$;
\item  $\sfm_X\colon  2^{2^X} \to  2^X $ sends each $\mathcal{F}\in2^{2^X}$ to the union of $\mathcal{F}$.   \end{itemize}
 \end{exmp}

If $F\colon\mathcal{A}\to\mathcal{B}$ is left adjoint to $G\colon\mathcal{B}\to\mathcal{A}$, then the adjunction $F\dashv G$ defines a monad   $$\mathbb{T}=(GF,G\varepsilon F,\eta)$$   in
$\mathcal{A}$, where $\eta$ and $\varepsilon$ are the unit and counit of the adjunction, respectively. %Every monad in $\mathcal{A}$ arises in this fashion.

\begin{exmp}(\cite[Example 2.11]{Manes2003}) The  contravariant powerset functor  on {\sf Set} is the functor \[\exp^{-1} \colon {\sf Set}^{\rm op}\to{\sf Set}\] that assigns to each   $f\colon X\to Y$ the map $$\exp^{-1} f=f^{-1}\colon 2^Y\to 2^X, \quad B\mapsto f^{-1}(B).$$ The functor $\exp^{-1}$ is right adjoint to its opposite \[(\exp^{-1})^{\rm op}\colon {\sf Set}\to{\sf Set}^{\rm op}.\]
The monad  defined by the adjunction $(\exp^{-1})^{\rm op} \dashv\exp^{-1}$ is called  the double contravariant powerset monad in {\sf Set}.\end{exmp}

 Let $\mathbb{T}=(T,\mu, \eta)$ be a monad in $\mathcal{A}$. A
$\mathbb{T}$-algebra is a pair $(A,h)$ consisting of an object $A$
 and a morphism
$h\colon TA\to A$ in $\mathcal{A}$   such that the
diagrams
$$\bfig \square[T^2A`TA`TA`A; Th`\mu_A`h`h]
\qtriangle(1100,0)/>`@{=}`>/[A`TA`A; \eta_A` `h] \efig$$ are
commutative. A morphism $f\colon (A,h)\to(A',f')$ of $\mathbb{T}$-algebras is a morphism $f\colon A\to A'$ of $\mathcal{A}$ that makes the square $$\bfig \square[TA`A`TA'`A'; h`Tf`f`h'] \efig$$
commutative. The category of $\mathbb{T}$-algebras and their
morphisms is   called the Eilenberg-Moore category  of the monad $\mathbb{T}=(T,\mu, \eta)$.

Assume that $F\colon\mathcal{A}\to\mathcal{B}$ is left adjoint to $G\colon\mathcal{B}\to\mathcal{A}$;  assume that     $\mathbb{T}=(GF,G\varepsilon F,\eta)$ is the   monad  in
$\mathcal{A}$ defined by the adjunction  $F\dashv G$.
Then, for each object $B$ of $\mathcal{B}$, the pair $(GB, G\varepsilon_B)$ is a $\mathbb{T}$-algebra and the assignment
$$ B\to^fB'~\mapsto~ (GB, G\varepsilon_B)\to^{Gf} (GB',
G\varepsilon_{B'}) $$ defines a functor  from $\mathcal{B}$ to the Eilenberg-Moore category of $\mathbb{T}$, called  the  comparison functor.
\begin{defn}(\cite{Riehl}) \begin{enumerate}[label={\rm(\roman*)}] \setlength{\itemsep}{0pt} \item %Assume that $F\colon\mathcal{A}\to\mathcal{B}$ is left adjoint to $G\colon\mathcal{B}\to\mathcal{A}$. We say that
An adjunction $F\dashv G$ is monadic if the comparison functor is an equivalence of categories.

\item
 A functor $G\colon\mathcal{B}\to\mathcal{A}$ is monadic if it admits a left adjoint that defines a monadic adjunction.
\item
 A category $\mathcal{B}$ is monadic over a category $\mathcal{A}$ if there
exists a functor $G\colon\mathcal{B}\to\mathcal{A}$ that is  monadic.
\end{enumerate}\end{defn}

\begin{exmp}(\cite{Manes76,Manes2003}) The Eilenberg-Moore category of the powerset monad $(\exp,\sfm,\sfe)$ in {\sf Set} is isomorphic to the category of complete lattices and join-preserving map.
An algebra for the double contravariant powerset monad in ${\sf Set}$ is  a complete and atomic Booelan algebra. The  functor   $\exp^{-1} \colon {\sf Set}^{\rm op}\to{\sf Set}$ is monadic, hence the category of sets is dually monadic over itself. \end{exmp}

%The notion of submonad is also needed in this paper.

Let $\mathbb{S}=(S,\sfm,\sfe)$ and $\mathbb{T}=(T,\mu,\eta)$   be monads in a category $\mathcal{A}$. A \emph{monad map} from $\mathbb{S}$ to $\mathbb{T}$ \cite{Manes2003} is a natural transformation $\kappa\colon S\to T$ for which  the diagrams \[\bfig\ptriangle[\id`S`T;\sfe`\eta`\kappa] \square(1100,0)[S^2`T^2`S`T;\kappa*\kappa`\sfm` \mu`\kappa]\efig\] are commutative, where $\kappa*\kappa$ stands for the horizontal composite of $\kappa$ with itself. If $\kappa$, as a morphism between functors, is a monomorphism, then we say that $\mathbb{S}=(S,\sfm,\sfe)$ is a \emph{submonad} of $\mathbb{T}=(T,\mu,\eta)$.

The covariant powerset monad in $\Set$ can be made into  a submonad of the double contravariant powerset monad in $\Set$,  see e.g. \cite[page 79]{Manes2003}.

\subsection*{Lifting of monads}
By a concrete category over   sets  \cite{AHS} we mean a pair $(\mathcal{A},U)$, where $\mathcal{A}$ is a category and $U\colon \mathcal{A}\to {\sf Set}$ is a faithful functor (called the forgetful functor). In the case that the forgetful functor is evident, we just say that $\mathcal{A}$ is a concrete category.
Together with the forgetful functor $$U\colon\setQ\to\Set, $$ the Goguen category $\setQ$ becomes a concrete category.

%We need the following   terminologies.

Let $(\mathcal{A},U)$ and $(\mathcal{B},V)$ be concrete categories over   sets.\begin{enumerate}[label={\rm(\roman*)}] \setlength{\itemsep}{0pt} \item  A functor $\mathscr{F}\colon \mathcal{A}\to \mathcal{B}$ is a lifting of a functor $F\colon\Set\to\Set$    if $F\circ U= V\circ \mathscr{F}$; that is, the square \[\bfig\square[\mathcal{A}`\mathcal{B}`\Set`\Set;\mathscr{F}`U`V`F]\efig\] commutes.

    Similarly, a functor $\mathscr{K}\colon \mathcal{A}^{\rm op}\to \mathcal{B}$ is a lifting of a functor $K\colon{\sf Set}^{\rm op}\to\Set$    if $K\circ U^{\rm op}= V\circ \mathscr{K}$.
\item A natural transformation $$\bfig\morphism|a|/@{->}@<5.5pt>/<500,0>[\mathcal{A}`\mathcal{B};\mathscr{F}] \morphism|b|/@{->}@<-5.5pt>/<500,0>[\mathcal{A}`\mathcal{B};\mathscr{G}] \morphism(250,90)|l|/=>/<0,-130>[`;\tau\thinspace]\efig$$    is   a lifting   of a natural transformation $$\bfig\morphism|a|/@{->}@<5.5pt>/<500,0>[\Set`\Set;F] \morphism|b|/@{->}@<-5.5pt>/<500,0>[\Set`\Set;G] \morphism(250,90)|l|/=>/<0,-130>[`;\kappa\thinspace]\efig$$   if  $\mathscr{F}$ is a lifting of $F$, $\mathscr{G}$ is a lifting of $G$, and $V\tau=\kappa U $. % (i.e., $U(\tau_{(X,\alpha)}) =\kappa_X$ for all $(X,\alpha)$ of $\setQ$).

\item (C.f. \cite[page 87]{Monoidal top}) A monad $(\mathscr{T},\mu,\eta)$ in   $\mathcal{A}$ is a lifting of a monad $(T,\sfm,\sfe)$ in   $\Set$  if  the functor $\mathscr{T}$ is a lifting of the functor $T$, and  the natural transformations $\mu$ and $\eta$ are lifting of $\sfm$ and $\sfe$, respectively. \end{enumerate}

\begin{prop}\label{lifting of monad} Let $(\mathcal{A},U)$ be a concrete category over   sets. Suppose that \begin{enumerate}[label={\rm(\roman*)}] \setlength{\itemsep}{0pt} \item $(T,\sfm,\sfe)$  is a monad in   $\Set$; \item $\mathscr{T}\colon \mathcal{A}\to \mathcal{A}$ is a lifting of   $T\colon\Set\to\Set$; \item $\mu\colon \mathscr{T}^2\to \mathscr{T}$ is a lifting of $\sfm\colon T^2\to T$ and $\eta\colon\id_\mathcal{A}\to\mathscr{T}$ is a lifting of $\sfe\colon\id_{\sf Set}\to T$. \end{enumerate} Then $(\mathscr{T},\mu,\eta)$ is a monad in $\mathcal{A}$ and it  is a lifting of $(T,\sfm,\sfe)$. \end{prop}

\begin{proof}We check  for example  the commutativity of the square: $$\bfig \square[\mathscr{T}^3`\mathscr{T}^2`\mathscr{T}^2`\mathscr{T};\mathscr{T}\mu`\mu \mathscr{T}`\mu`\mu]\efig$$ For each object $A$ of $\mathcal{A}$,  since \begin{align*}U((\mu\circ \mathscr{T}\mu)_A)&=U(\mu_A\circ \mathscr{T}\mu_A)\\ &= U(\mu_A)\circ U(\mathscr{T}\mu_A)\\ &= \sfm_{U(A)}\circ T\sfm_{U(A)}\\ &= (\sfm\circ T\sfm)_{U(A)} \\ &= (\sfm\circ\sfm T)_{U(A)}\\ &= U((\mu\circ\mu \mathscr{T})_A) \end{align*} then $\mu\circ \mathscr{T}\mu= \mu\circ\mu \mathscr{T}$ because $U$ is faithful. \end{proof}

\section{The powerset monad in $\setQ$}
This section recalls the construction of the powerset monad $(\mathscr{U},\sfm,\sfe)$ in $\setQ$, which  first appeared in \cite[Section 4]{EKS12} under the name \emph{unbalanced powerobject monad}. In next section, we shall see that for a commutative quantale,  $(\mathscr{U},\sfm,\sfe)$ is a submonad of the double contravariant powerset monad in $\setQ$.

For each object $(X,\alpha)$ of $\setQ$, let \[\mathscr{U}(X,\alpha)=(\sQ^X,\alpha^\downarrow),\] where for all $\gamma\in\sQ^X$, $$\alpha^\downarrow(\gamma)=\alpha\swarrow\gamma=\bw_{x\in X}\alpha(x)\ldd\gamma(x).$$   For each Goguen map $f\colon(X,\alpha)\to(Y,\beta)$, the map \[\mathscr{U} f\colon(\sQ^X,\alpha^\downarrow)\to (\sQ^Y,\beta^\downarrow), \quad \gamma\mapsto f(\gamma)\] satisfies the Goguen condition since \[\alpha^\downarrow(\gamma)=\alpha\swarrow\gamma\leq f(\alpha)\swarrow f(\gamma)\leq \beta\swarrow f(\gamma)=\beta^\downarrow(f(\gamma)).\] Thus, the assignment $f\mapsto \mathscr{U} f$ defines a functor \[\mathscr{U}\colon\setQ\to\setQ,\] called the (covariant) \emph{powerset functor on $\setQ$}.

%\begin{rem}\label{power}
%As argued in   \cite{HS1991,Stout92}, in order to ``obtain  category with a good internal logic which captures fuzzy subsets we should turn our attention to the unbalanced subobjects''. An unbalanced subobject of an object $(X, \alpha)$   is an equivalence class of monomorphisms in $\setQ$ which are also epimorphisms. If we identify a set $X$ with $(X,1_X)$, where $1_X$ is a the constant fuzzy set with value $1\in\sQ$, then an  unbalanced subobject  of $X$ is essentially a fuzzy set $\gamma\colon X\to\sQ$. So, the fuzzy set $(\sQ^X,\alpha^\downarrow)$ is an ``internal fuzzy powerobject'' of $(X,\alpha)$ \cite[Section 3]{EKS12}, with the membership degree $\alpha^\downarrow(\gamma)$ interpreted as the degree  that $\gamma$ is an unbalanced subobject  of $(X, \alpha)$.  %which accounts for the term the \emph{covariant internal unbalanced powerobject functor} in  \cite{EKS12} for the functor $\mathscr{U}$. This is in  concordance with the   \emph{Principle of Fuzzification} of Goguen  \cite[Section 4]{Goguen67}: A fuzzy   something is a fuzzy set of somethings.

%From a similar yet somewhat unfamiliar viewpoint, the fuzzy set $(\sQ^X,\alpha^\uparrow)$ can also be understood as a ``fuzzy powerset'' of $(X,\alpha)$, with membership degree of $\gamma$ being $\gamma\swarrow\alpha$, the degree that $\alpha$ is an unbalanced subobject of $(X, \gamma)$. This explains why both   $\mathscr{P}$  and $\mathscr{P}^\dag$  in the previous section are called  contravariant powerset functor, hence the name of the monad $(\mathfrak{P},\mu,\eta)$  in $\setQ$.
%\end{rem}

\begin{lem}For each object $(X,\alpha)$ of $\setQ$,  both $$\sfe_{(X,\alpha)}\colon (X,\alpha)\to (\sQ^X,\alpha^\downarrow),\quad \sfe_{(X,\alpha)}(x) =k_x $$ and \[\sfm_{(X,\alpha)}\colon (\sQ^{\sQ^X},\alpha^{\downarrow\downarrow})\to (\sQ^X,\alpha^\downarrow),\quad \sfm_{(X,\alpha)}(\Lambda)= \bv_{\gamma\in\sQ^X}\Lambda(\gamma)\with\gamma\]  are Goguen maps. \end{lem}
\begin{proof} The conclusion is   contained in \cite{EKS12},  the verification is included here for convenience of the reader.

For each $x\in X$, \[\alpha^\downarrow(\sfe_{(X,\alpha)}(x)) = \alpha\swarrow k_x= \alpha(x)\ldd k=\alpha(x), \] hence $\sfe_{(X,\alpha)}$ is a Goguen map.

For each $\Lambda\in\sQ^{\sQ^X}$, \begin{align*}\alpha^{\downarrow\downarrow}(\Lambda)&= \bw_{\gamma\in\sQ^X}(\alpha\swarrow\gamma)\ldd\Lambda(\gamma) \\ &= \bw_{x\in X}\bw_{\gamma\in\sQ^X}(\alpha(x)\ldd\gamma(x))\ldd\Lambda(\gamma)\\ &= \bw_{x\in X} \Big(\alpha(x)\ldd\bv_{\gamma\in\sQ^X}\Lambda(\gamma)\with\gamma(x) \Big) \\ &= \alpha^\downarrow(\sfm_{(X,\alpha)}(\Lambda)), \end{align*} hence $\sfm_{(X,\alpha)}$ is a Goguen map. \end{proof}

The triple \[ (\mathscr{U},\sfm,\sfe) \] is a monad in $\setQ$. Instead of verifying directly that $\sfe$ and $\sfm$ are  natural transformations and satisfy the monad requirements, we show that  it is a lifting of a monad in $\Set$, namely, the (covariant) $\sQ$-powerset monad   described below.

Assigning to each   $f\colon X\to Y$  the map \[\exp_\sQ f\colon \sQ^X\to \sQ^Y, \quad  \gamma\mapsto f(\gamma) \] defines a functor \[\exp_\sQ\colon{\sf Set}\to{\sf Set},\] called the (covariant) \emph{$\sQ$-powerset functor}.

The functor $\exp_\sQ$ gives rise to a monad \[ (\exp_\sQ,\sfm,\sfe)\] in the category of sets \cite{AM1975,Machner,Manes1982},  where for each set $X$,  \begin{itemize}  \item  $\sfe_X\colon X\to \sQ^X$ is the map such that for all $x\in X$, $\sfe_X(x)=k_x$;
\item  $\sfm_X\colon  \sQ^{\sQ^X} \to  \sQ^X $ is the map such that  for all $\Lambda\in\sQ^{\sQ^X}$ and $x\in X$, \[\sfm_X(\Lambda)(x)= \bv_{\gamma\in\sQ^X}\Lambda(\gamma)\with\gamma(x).\]   \end{itemize}

When $\sQ$ is the Boolean algebra $2=\{0,1\}$, the monad  $(\exp_\sQ,\sfm,\sfe)$ is then the  powerset monad $(\exp,\sfm,\sfe)$ in Example \ref{powerset monad}. Thus, for a general unital quantale $\sQ$, we call $(\exp_\sQ,\sfm,\sfe)$   the   $\sQ$-powerset monad in $\Set$.

For each set $X$, define $\kappa_X\colon 2^X\to\sQ^X$ by \[\kappa_X(A)(x)=\begin{cases}k& x\in A,\\ 0&x\notin A.\end{cases}\] Then $\kappa=\{\kappa_X\}_X$ is a monad map, exhibiting $(\exp,\sfm,\sfe)$ as a submonad of  $(\exp_\sQ,\sfm,\sfe)$.

It is clear that \begin{itemize}\item the functor  $\mathscr{U}$ is a lifting of the functor $\exp_\sQ$; \item   for each object $(X,\alpha)$ of $\setQ$, $U(\sfm_{(X,\alpha)})  =\sfm_X$; \item for each object $(X,\alpha)$ of $\setQ$,  $ U(\sfe_{(X,\alpha)})  =\sfe_X$. \end{itemize} Then by Proposition \ref{lifting of monad},   $(\mathscr{U},\sfm,\sfe)$ is a monad in $\setQ$,  a lifting of $(\exp_\sQ,\sfm,\sfe)$. Because of this fact, we   call $(\mathscr{U},\sfm,\sfe)$ the \emph{powerset monad in $\setQ$}, instead of the \emph{unbalanced powerobject monad} as in \cite{EKS12}.

\begin{rem}The  monad $(\exp_\sQ,\sfm,\sfe)$ can be lifted to a monad in $\setQ$ in different ways,   one lifting different from $\mathscr{U}$ is given below.

For each object $(X,\alpha)$ of $\setQ$, define $\alpha^\circ\colon\sQ^X\to\sQ$ by \[\alpha^\circ(\gamma)=\bv_{x\in X} \gamma(x)\with\alpha(x).\] The assignment $(X,\alpha)\mapsto(\sQ^X,\alpha^\circ)$ yields a functor \[ \mathscr{W}\colon\setQ\to\setQ.\] Both $\sfe_X\colon (X,\alpha)\to (\sQ^X, \alpha^\circ)$ and $\sfm_X\colon (\sQ^{\sQ^X}, \alpha^{\circ\circ})\to (\sQ^X, \alpha^\circ)$ are Goguen maps (verifications are left to the reader), so  the triple \[(\mathscr{W},\sfm,\sfe)\] is a monad in $\setQ$ and it is also a lifting of $(\exp_\sQ,\sfm,\sfe)$.
\end{rem}

%The Kleisli category of the monad $(\mathscr{U},\sfm,\sfe)$ has been described in \cite{EKS12}.

The fact that $(\mathscr{U},\sfm,\sfe)$ is a lifting of $(\exp_\sQ,\sfm,\sfe)$ is very useful.  In the following we use this fact to determine the Eilenberg-Moore algebras of $(\mathscr{U},\sfm,\sfe)$.
Let $((X,\alpha),h)$ be an algebra of the monad $(\mathscr{U},\sfm,\sfe)$.  By definition $h\colon (\sQ^X,\alpha^\downarrow)\to(X,\alpha)$ is a Goguen map. Since $(\mathscr{U},\sfm,\sfe)$ is a lifting of $(\exp_\sQ,\sfm,\sfe)$, it is readily verified that $(X,h)$ is  an algebra of the monad $(\exp_\sQ,\sfm,\sfe)$. Conversely, if $(X,h)$ is an algebra of the monad $(\exp_\sQ,\sfm,\sfe)$ and $h\colon (\sQ^X,\alpha^\downarrow)\to(X,\alpha)$ satisfies the Goguen condition, then $((X,\alpha),h)$ is an algebra of the monad $(\mathscr{U},\sfm,\sfe)$. Therefore, in order to determine algebras of the monad  $(\mathscr{U},\sfm,\sfe)$, we need to determine algebras of  $(\exp_\sQ,\sfm,\sfe)$ first.

The algebras of $(\exp_\sQ,\sfm,\sfe)$ have been determined in \cite{Machner,PT1989}. These algebras can be described either as cocomplete $\sQ$-lattices or as $\sQ$-modules. A sketch of the ideas is included here for convenience of the reader.

A $\sQ$-order on a set $X$ is a map  $o\colon X\times X\to\sQ$   such that  for all $x,y,z\in X$, $$k\leq o(x,x) \quad \text{and}\quad  o(y,z)\with o(x,y)\leq o(x,z). $$ The pair $(X,o)$ is called a $\sQ$-ordered set or a $\sQ$-category \cite{EGHK,Stubbe2005}.

A map $f\colon(X,o_X)\to(Y,o_Y)$ between $\sQ$-ordered sets is said to preserve  $\sQ$-order, if for all $x_1,x_2\in X$, $$o_X(x_1,x_2)\leq o_Y(f(x_1),f(x_2)).$$ A $\sQ$-order-preserving map $f\colon(X,o_X)\to(Y,o_Y)$ is a left adjoint, if there is a $\sQ$-order-preserving map $g\colon(Y,o_Y)\to(X,o_X)$ such that for all $x\in X$ and $y\in Y$, \[o_Y(f(x),y)=o_X(x,g(y)).\]

Let $(X,o)$ be a $\sQ$-ordered set, $a\in X$, and $\gamma\in\sQ^X$. We say that \begin{itemize}\item  $(X,o)$ is \emph{separated}  if   $x=y$ whenever $k\leq o(x,y)\wedge o(y,x)$. \item   $a$  is a  \emph{supremum}  of $\gamma$  if  for all $y\in X$, $o(a,y)=\bw_{x\in X} o(x,y)\ldd\gamma(x).  $  \item $(X,o)$ is \emph{cocomplete} if  every   $\gamma\in\sQ^X$ has a supremum. \item $(X,o)$  is   a \emph{cocomplete $\sQ$-lattice} if it is both separated and cocomplete.\end{itemize}

%It is clear that in a separated $\sQ$-ordered set $(X,o)$, the supremum of a fuzzy set, when exists, is unique.

\begin{defn} (\cite{Joyal-Tierney}) A   $\sQ$-module (precisely, a left $\sQ$-module) is a   pair $(X,\otimes)$, where $X$ is a complete lattice and $\otimes\colon \sQ\times X \to X$ is a   map, called a (left) $\sQ$-action on $X$,  subject to the following conditions:   for all $x\in X$ and $r,s\in \sQ$, \begin{enumerate}[label={\rm(\roman*)}] \setlength{\itemsep}{0pt}
  \item $k\otimes  x=x$, where $k$ is the unit of $\sQ$; \item $ s\otimes (r\otimes  x)  =(s\with r)\otimes x$; \item   $r\otimes -\colon X\to X$ preserve joins;   \item $-\otimes x\colon \sQ\to X$ preserve joins. \end{enumerate}A   homomorphism  $f\colon (X,\otimes)\to(Y,\otimes)$ between $\sQ$-modules is a join-preserving map $f\colon X\to Y$   that preserves the action, i.e., $ r\otimes f(x)= f(r\otimes x)$ for all $r\in \sQ$ and $x\in X$.\end{defn}

\begin{exmp}(\cite{Joyal-Tierney}) \label{QX as module} For each set $X$,  define  $\otimes\colon\sQ\times\sQ^X\to\sQ^X$ by $(r\otimes\gamma)(x)= r\with\gamma(x)$, then $(\sQ^X,\otimes)$ is a $\sQ$-module.\end{exmp}

Given a $\sQ$-module $(X,\otimes)$, define $o\colon X\times X\to\sQ$ by  \[o(x,y)=\bv\{r\in\sQ\mid r\otimes x\leq y\}. \]  Then $(X,o)$ is a cocomplete $\sQ$-lattice with supremum of $\gamma\in\sQ^X$ given by $$\sup\gamma=\bv_{x\in X}\gamma(x)\otimes x.$$  Conversely, given a cocomplete $\sQ$-lattice $(X,o)$, define a binary relation $\leq$ on $X$ by letting $x\leq y$ if $k\leq o(x,y)$. Then $(X,\leq)$ is a complete lattice. Furthermore,   the assignment  $(r,x)\mapsto \sup r_x$   defines  a $\sQ$-action on the complete lattice $(X,\leq)$. These   processes are inverse to each other, hence the category of cocomplete $\sQ$-lattices and left adjoints is isomorphic to that of $\sQ$-modules and $\sQ$-module homomorphisms. Details of these claims can be found in \cite[Section 4]{Stubbe2006} or \cite[Section 3.3]{EGHK}.

Let $(X,o)$ be cocomplete $\sQ$-lattice. Then $X$ together with the map  $\sup\colon \sQ^X\to X$  is an algebra of the monad $(\exp_\sQ,\sfm,\sfe)$. Conversely, let $(X,h)$ be an algebra of $(\exp_\sQ,\sfm,\sfe)$. Since the powerset monad $(\exp,\sfm,\sfe)$ is a submonad of $(\exp_\sQ,\sfm,\sfe)$, $X$ together with the restriction of $h$ on $2^X$ is an algebra of $(\exp,\sfm,\sfe)$, hence $X$ is a complete lattice and  $h$  maps each subset of $X$ to its join \cite[page 142]{MacLane1998}.  Define $\otimes\colon \sQ\times X\to X$ by \(r\otimes x=h(r_x).\)   Then $(X,\otimes)$ is a $\sQ$-module. Therefore, an algebra of the monad $(\exp_\sQ,\sfm,\sfe)$ is essentially a  cocomplete $\sQ$-lattice, or equivalently, a $\sQ$-module.

Now we are able to describe algebras of the monad $(\mathscr{U},\sfm,\sfe)$.

\begin{prop}\label{as modules}  An algebra of the monad $(\mathscr{U},\sfm,\sfe)$ is a fuzzy set $\alpha\colon X\to\sQ$ of a cocomplete $\sQ$-lattice  $(X, o)$ such that the map $$\sup\colon (\sQ^X,\alpha^\downarrow)\to(X,\alpha)$$  satisfies the Goguen condition; that is, for all $\gamma\in\sQ^X$, \[ \alpha\swarrow\gamma\leq \alpha(\sup\gamma).\] A  homomorphism is a map $f\colon X\to Y$ that is simultaneously a Goguen map $f\colon(X,\alpha)\to(Y,\beta)$ and a left adjoint $f\colon(X,o_X)\to(Y,o_Y)$. \end{prop}

The value $\alpha\swarrow\gamma$ can be viewed as the degree that the fuzzy set $\gamma$ is contained in the fuzzy set $\alpha$ (c.f. \cite[page 369]{Goguen69}),   the inequality $$\alpha\swarrow\gamma\leq \alpha(\sup\gamma)$$ says that $\alpha$ is closed under formation of suprema in the cocomplete $\sQ$-lattice $(X,o)$. So, an algebra of  $(\mathscr{U},\sfm,\sfe)$ is a fuzzy set of a cocomplete $\sQ$-lattice that is closed under formation of suprema.

In terms of $\sQ$-modules, we have:  \begin{prop}\label{as modules}  An algebra of  $(\mathscr{U},\sfm,\sfe)$ is a fuzzy set $\alpha\colon X\to\sQ$ of a $\sQ$-module  $(X, \otimes)$ such that \begin{enumerate}[label={\rm(\roman*)}] \setlength{\itemsep}{0pt}\item for each subset $A$ of $X$, $\bw_{x\in A}\alpha(x)\leq\alpha(\bv A)$; \item for each $r\in\sQ$ and   $x\in X$, $\alpha(x)\ldd r\leq \alpha(r\otimes x)$.\end{enumerate} A  homomorphism is a map $f\colon X\to Y$ that is simultaneously a Goguen map $f\colon(X,\alpha)\to(Y,\beta)$ and a $\sQ$-module homomorphism $f\colon(X,\otimes_X)\to(Y,\otimes_Y)$. \end{prop}

%For each object $(X,\alpha)$ of $\setQ$, $(\sQ^X,\alpha^\downarrow)$ is an algebra of  $(\mathscr{U},\sfm,\sfe)$, where  $\sQ^X$ is viewed as a $\sQ$-module as in Example \ref{QX as module}.  This is because \begin{itemize}\item for each subset $\{\gamma_i\}_i$ of $\sQ^X$, $\bw_i\alpha^\downarrow(\gamma_i)=\alpha^\downarrow(\bv_i \gamma_i)$; \item for each $r\in\sQ$ and   $\gamma\in \sQ^X$, $$\alpha^\downarrow(\gamma)\ldd r= \bw_{x\in X}(\alpha(x)\ldd\gamma(x))\ldd r=\bw_{x\in X}\alpha(x)\ldd(r\with\gamma(x)) =  \alpha^\downarrow(r\otimes \gamma).$$\end{itemize}

%The assignment $(X,\alpha)\mapsto (\sQ^X,\alpha^\downarrow)$ defines a left adjoint of the forgetful functor from the category of $\mathscr{U}$-algebras to the Goguen category $\setQ$.

\section{The double contravariant powerset monad in $\setQ$}
For each object $(X,\alpha)$ of $\setQ$, let \[\mathscr{P} (X,\alpha)=(\sQ^X, \alpha^\uparrow)\quad\text{and}\quad \mathscr{P}^\dag (X,\alpha)=(\sQ^X, \alpha_\dag^\uparrow),\] where  for all $\gamma\in\sQ^X$, $$\alpha^\uparrow(\gamma)=\gamma\swarrow\alpha\quad\text{and}\quad \alpha_\dag^\uparrow(\gamma)=\alpha\searrow\gamma.$$  %For each Goguen map $f\colon(X,\alpha)\to(Y,\beta)$  and  each $\lambda\in\sQ^Y$, let both $\mathscr{P} f(\lambda)$ and $\mathscr{P}^\dag f(\lambda)$ be the inverse image of $\lambda$; that is,  \[\mathscr{P} f(\lambda) =\lambda\circ f =\mathscr{P}^\dag f(\lambda).\]

\begin{lem} For each Goguen map $f\colon(X,\alpha)\to(Y,\beta)$, both \[\mathscr{P} f\colon(\sQ^Y,\beta^\uparrow)\to (\sQ^X,\alpha^\uparrow),\quad \lambda\mapsto \lambda\circ f\] and \[\mathscr{P}^\dag f\colon(\sQ^Y,\beta_\dag^\uparrow)\to (\sQ^X,\alpha_\dag^\uparrow),\quad \lambda\mapsto \lambda\circ f\] satisfy the Goguen condition. \end{lem}

\begin{proof} We verify the case of $\mathscr{P} f$ for example.
Since $f\colon(X,\alpha)\to(Y,\beta)$ is a Goguen map, then $f(\alpha)\leq \beta$, hence by Lemma \ref{image vs preimage}, \[ \lambda\swarrow\beta \leq \lambda\swarrow f(\alpha) = (\lambda\circ f)\swarrow\alpha \] for all $\lambda\in\sQ^Y$, which shows that $\mathscr{P} f\colon(\sQ^Y,\beta^\uparrow)\to (\sQ^X,\alpha^\uparrow)$ satisfies the Goguen condition.\end{proof}

Therefore, we obtain two contravariant functors: \[\mathscr{P}\colon \setQ^{\rm op}\to\setQ \] and   $$\mathscr{P}^\dag \colon \setQ\to\setQ^{\rm op}.$$

\begin{prop}\label{adjunction P Pdagger}  $\mathscr{P}\colon \setQ^{\rm op}\to\setQ $ is right adjoint to $\mathscr{P}^\dag \colon \setQ\to\setQ^{\rm op}$.    \end{prop}

\begin{proof}It suffices to check that  $$f\colon(X,\alpha)\to(\sQ^Y,\beta^\uparrow)$$ is a Goguen map if and only if so too is its transpose \[\overline{f}\colon(Y,\beta)\to(\sQ^X,\alpha_\dag^\uparrow), \quad \overline{f}(y)(x)=f(x)(y).\]   This is easy since \begin{align*} &\quad\quad\quad \text{$f\colon(X,\alpha)\to(\sQ^Y,\beta^\uparrow)$ is a Goguen map}\\ &\iff \forall x\in X, \forall y\in Y,~ \alpha(x)\leq   f(x)(y)\ldd \beta(y) \\ &\iff \forall x\in X, \forall y\in Y,~ \alpha(x)\with\beta(y)\leq   f(x)(y)  \\ &\iff \forall y\in Y,\forall x\in X, ~ \beta(y)\leq \alpha(x)\rdd \overline{f}(y)(x) \\ &\iff \overline{f}\colon(Y,\beta)\to(\sQ^X,\alpha_\dag^\uparrow)~\text{is a Goguen map}.\qedhere \end{align*} \end{proof}

In the adjunction  $\mathscr{P}^\dag \dashv\mathscr{P}$: \begin{itemize}\item the unit $\eta$ assigns to  each object  $(X,\alpha)$ of $\setQ$ the map \[\eta_{(X,\alpha)}\colon (X,\alpha)\to \mathscr{P} \mathscr{P}^\dag(X,\alpha),\quad \eta_{(X,\alpha)}(x)(\gamma)=\gamma(x);  \]  \item  the counit $\epsilon$ assigns to  each object $(Y,\beta)$ of $\setQ^{\rm op}$  the opposite of $$\epsilon_{(Y,\beta)}\colon (Y,\beta) \to \mathscr{P}^\dag \mathscr{P} (Y,\beta),\quad \epsilon_{(Y,\beta)}(y)(\lambda)=\lambda(y). $$  \end{itemize}

We call the monad defined by the adjunction $$\mathscr{P}^\dag \dashv\mathscr{P}$$  the \emph{double contravariant powerset monad} in $\setQ$ and  denote it by \[\mathfrak{P}=(\mathscr{P} \mathscr{P}^\dag,\mu,\eta).\]

As usual,   we   write $\mathfrak{P}$ for both the  monad $(\mathscr{P} \mathscr{P}^\dag,\mu,\eta)$ and the  functor $\mathscr{P} \mathscr{P}^\dag$. We spell out   the details of the monad $\mathfrak{P}$ for later use. For each object $(X,\alpha)$ of $\setQ$, \begin{itemize}
\item $\mathfrak{P}(X,\alpha)=(\sQ^{\sQ^X},(\alpha_\dag^{\uparrow})^\uparrow)$, where for all $\Lambda\colon\sQ^X\to \sQ$, $$(\alpha_\dag^{\uparrow})^\uparrow(\Lambda)= \bw_{\gamma\in\sQ^X}(\Lambda(\gamma)\ldd(\alpha\searrow\gamma));$$
%= \bw_{\gamma\in\sQ^X}\Big[\Lambda(\gamma)\ldd\Big(\bw_{x\in X}\alpha(x)\rdd\gamma(x)\Big)\Big]

 \item the   unit $\eta$ assigns to  $(X,\alpha)$ the Goguen map $$\eta_{(X,\alpha)}\colon (X,\alpha)\to (\sQ^{\sQ^X},(\alpha_\dag^{\uparrow})^\uparrow)$$ given by $$\eta_{(X,\alpha)}(x)(\gamma)=\gamma(x)$$ for all $x\in X$ and $\gamma\in\sQ^X$;
%That $\eta_{(X,\alpha)}$ satisfies the Goguen condition follows from $$\alpha^{\uparrow\uparrow}(\eta_{(X,\alpha)}(x))= \bw_{\gamma\in\sQ^X}(\sub_X(\alpha,\gamma)\ra \gamma(x))\geq\alpha(x).$$

\item the   multiplication $\mu$ assigns to $(X,\alpha)$ the Goguen  map $$\mu_{(X,\alpha)}\colon \mathfrak{P}^2(X,\alpha)\to \mathfrak{P}(X,\alpha)$$ given by  \[\mu_{(X,\alpha)}(\mathbb{H})(\gamma) =\mathbb{H}(\widehat{\gamma}),\quad  \widehat{\gamma}(\Lambda)=\Lambda(\gamma)\]  for all $\mathbb{H}\colon \sQ^{\sQ^{\sQ^X}}\to\sQ$, $\gamma\in\sQ^X$  and $\Lambda\in\sQ^{\sQ^X}$.
\end{itemize}

The monad $\mathfrak{P}$  is a lifting of a monad in the category of sets, namely, a lifting of the double contravariant $\sQ$-powerset monad that we describe now.

By the \emph{contravariant $\sQ$-powerset functor} on {\sf Set} we mean the functor \[\exp_\sQ^{-1}\colon {\sf Set}^{\rm op}\to{\sf Set}\] that sends a   map $f\colon X\to Y$ to  $$f^{-1}\colon \sQ^Y\to \sQ^X, \quad \lambda\mapsto \lambda\circ f.$$
%The notation $\exp_\sQ^{-1}$ is inspired by the notation $P^{-1}$ in \cite{Manes2003}.

The contravariant $\sQ$-powerset functor  $\exp_\sQ^{-1}$ is right adjoint to its opposite \[(\exp_\sQ^{-1})^{\rm op}\colon {\sf Set}\to{\sf Set}^{\rm op}.\]  In the adjunction  $(\exp_\sQ^{-1})^{\rm op} \dashv\exp_\sQ^{-1}$, \begin{itemize}\item  the unit $\eta$ assigns to each set $X$    the   map \[\eta_X\colon X\to \sQ^{\sQ^X},\quad \eta_X(x)(\gamma)=\gamma(x); \] \item the counit $\epsilon$ assigns to each set $Y$ the map $$\epsilon_Y\colon Y \to \sQ^{\sQ^Y},\quad \epsilon_Y(y)(\lambda)=\lambda(y).$$ \end{itemize}

The   monad \[  (\exp_\sQ^{-2},\mu,\eta) \]  defined by the adjunction $(\exp_\sQ^{-1})^{\rm op} \dashv\exp_\sQ^{-1}$ is called the \emph{double contravariant $\sQ$-powerset monad}  (c.f. \cite[Remark 1.2.7]{Hoehle2001}) in {\sf Set}.\footnote{The double contravariant $\sQ$-powerset functor $\exp_\sQ^{-2}$ already appeared in   \cite[page 112]{EG92}.} When $\sQ$ is the Boolean algebra $\{0,1\}$, this monad  is just the double contravariant powerset monad in {\sf Set}.

We spell out details of the monad $(\exp_\sQ^{-2},\mu,\eta)$ for later use:  \begin{itemize}\item the functor $\exp_\sQ^{-2}$ assigns to each  $f\colon X\to Y$   the map $$\exp_\sQ^{-2}f\colon \sQ^{\sQ^X}\to \sQ^{\sQ^Y}$$ given by \begin{equation}\label{P-2f} \exp_\sQ^{-2}f(\Lambda)(\gamma)=\Lambda(\gamma\circ f) \end{equation} for all $\Lambda\in\sQ^{\sQ^X}$ and $\gamma\in\sQ^Y$; \item the unit $\eta$ assigns to each set $X$ the map  $\eta_X\colon X\to \sQ^{\sQ^X}$ given by $\eta_X(x)(\gamma)=\gamma(x);$
\item the multiplication $\mu$ assigns to each set $X$ the map $$\mu_X\colon \exp_\sQ^{-4}(X)\to \exp_\sQ^{-2}(X)$$ given by \begin{equation}\label{def of mu}\mu_X(\mathbb{H})(\gamma) =\mathbb{H}(\widehat{\gamma}),\quad  \widehat{\gamma}(\Lambda)=\Lambda(\gamma)\end{equation}  for all $\mathbb{H}\colon \sQ^{\sQ^{\sQ^X}}\to\sQ$, $\gamma\in\sQ^X$  and $\Lambda\in\sQ^{\sQ^X}$. \end{itemize}

It is clear that  \begin{itemize}\item
the functor  $\mathscr{P}\colon \setQ^{\rm op}\to\setQ$ is a lifting of $\exp_\sQ^{-1}\colon {\sf Set}^{\rm op}\to{\sf Set}$; \item     the functor $\mathscr{P}^\dag\colon \setQ\to\setQ^{\rm op}$ is a lifting of  $(\exp_\sQ^{-1})^{\rm op}\colon {\sf Set}\to{\sf Set}^{\rm op}$;
\item
the multiplication and the unit of   the monad $(\mathfrak{P},\mu,\eta)$ are  lifting of that of the monad $(\exp_\sQ^{-2},\mu,\eta)$, respectively.  \end{itemize}
Therefore,  the monad $(\mathfrak{P},\mu,\eta)$ is a lifting of $(\exp_\sQ^{-2},\mu,\eta)$.

\begin{rem} \begin{enumerate}[label={\rm(\roman*)}] \setlength{\itemsep}{0pt}
\item Though  $\mathscr{P}$ is a lifting of $\exp_\sQ^{-1}$ and $\mathscr{P}^\dag$ is a lifting of  $(\exp_\sQ^{-1})^{\rm op}$, the functor $\mathscr{P}^\dag$ is \emph{not} the opposite of   $\mathscr{P}$ unless the quantale $\sQ$ is commutative.
\item The construction of the adjunction $\mathscr{P}^\dag\dashv\mathscr{P}$, hence that of the monad $\mathfrak{P}$, makes use of the quantale structure of $\sQ$. But, the construction of the adjunction $(\exp_\sQ^{-1})^{\rm op} \dashv\exp_\sQ^{-1}$ does not depend on the quantale structure of $\sQ$; that means, $\sQ$ can be replaced by any nonempty set in this construction.
 \item The functors $\exp_\sQ\colon{\sf Set}\to{\sf Set}$ and  \(\exp_\sQ^{-1}\colon {\sf Set}^{\rm op}\to{\sf Set}\) are closely related to each other. Of particular interest is the following fact for which the verification is left to the reader:   for any pullback square in the category of sets, as displayed on the left,
$$\bfig
\square[A`C`B`D;f`h`j`g] \square(1100,0)[\sQ^C`\sQ^A`\sQ^D`\sQ^B;\exp_\sQ^{-1} f`\exp_\sQ j`\exp_\sQ h
`\exp_\sQ^{-1} g]
\efig$$
the right square is commutative.
In the case that $\sQ$ is the Boolean algebra $2=\{0,1\}$, this fact is just the Beck-Chevalley condition of the category of sets (see e.g. \cite[page 179]{Riehl}).

It should be warned that though  $\mathscr{U}$ is a lifting of $\exp_\sQ$ and $\mathscr{P}$ is a lifting of $\exp_\sQ^{-1}$, the nice connection between $\exp_\sQ$ and $\exp_\sQ^{-1}$ does not carry over.  For instance, it does not make sense to formulate a  square for $\mathscr{U}$ and $\mathscr{P}$ as displayed on the right for $\exp_\sQ$ and $\exp_\sQ^{-1}$, since  $\mathscr{U}$ sends an object $(X,\alpha)$ of $\setQ$  to $(\sQ^X,\alpha^\downarrow)$, while   $\mathscr{P}$ sends it to $(\sQ^X,\alpha^\uparrow)$. \end{enumerate}  \end{rem}

The following theorem implies that  the category of Eilenberg-Moore algebras of the monad $(\mathfrak{P},\mu,\eta)$ is equivalent to the opposite category of $\setQ$, hence  $\setQ$ is dually monadic over itself.

\begin{thm}The functor $\mathscr{P}\colon \setQ^{\rm op}\to\setQ $ is monadic. \end{thm}
\begin{proof}We apply the ``reflexive tripleability theorem'' (see e.g. \cite[Proposition 5.5.8]{Riehl})  to prove the conclusion. Since $\setQ$ is a complete category,    we only need to show that  the functor $$\mathscr{P}\colon \setQ^{\rm op}\to\setQ $$
reflects isomorphisms and preserves coequalizers of reflexive pairs.\footnote{A parallel pair of morphisms $r,s: A\to B$ in a category is reflexive if they have a common right inverse; that is, there is a morphism $i: B\to A$ such that $r\circ i=1_B=s\circ i$.}

Suppose that $f\colon(X,\alpha)\to(Y,\beta)$ is a Goguen map such that \[\mathscr{P}f\colon (\sQ^Y,\beta^\uparrow)\to (\sQ^X,\alpha^\uparrow),\quad \lambda\mapsto \lambda\circ f\] is an isomorphism in $\setQ$. Then $f$ is a bijection and  \[ \lambda\swarrow\beta=(\lambda\circ f)\swarrow\alpha \] for all $\lambda\in\sQ^Y$.  Putting $\lambda=\alpha\circ f^{-1}$ gives that $\beta\leq\alpha\circ f^{-1}$, hence $\beta\circ f\leq\alpha$ and consequently $\beta\circ f=\alpha$. Therefore, $f$ is an isomorphism in $\setQ$, hence an isomorphism in $\setQ^{\rm op}$. This proves that $\mathscr{P}$ reflects isomorphisms.

Now we show that $\mathscr{P}$ preserves coequalizers of reflexive pairs. Consider a coequalizer in $\setQ^{\rm op}$ of a   reflexive pair; this means we have an equalizer  \[\bfig  \morphism(0,0)<500,0>[ (Z,\gamma)`(X,\alpha);e]\morphism(500,0)|a|/@{->}@<3pt>/<500,0>[(X,\alpha) `(Y,\beta); f]
\morphism(500,0)|r|/@{->}@<-3pt>/<500,0>[(X,\alpha) `(Y,\beta); g]\efig\]  in  $\setQ$ together with a Goguen map $h\colon (Y,\beta)\to (X,\alpha)$ such that both $h\circ f$ and $h\circ g$ are the identity map on $(X,\alpha)$. We wish to prove that \[\bfig \morphism(0,0)|a|/@{->}@<3pt>/<600,0>[\mathscr{P}(Y,\beta) `\mathscr{P}(X,\alpha); \mathscr{P}f]
\morphism(0,0)|r|/@{->}@<-3pt>/<600,0>[\mathscr{P}(Y,\beta) `\mathscr{P}(X,\alpha); \mathscr{P}g] \morphism(600,0)<600,0>[ \mathscr{P}(X,\alpha)`\mathscr{P}(Z,\gamma);\mathscr{P}e]\efig\]   is a coequalizer in $\setQ$.

Since $h$ is a common left inverse for $f$ and $g$, then \begin{itemize}\item both $f$ and $g$ are injective; \item $\beta\circ f=\alpha=\beta\circ g$; \item for all $x_1,x_2\in X$,   $f(x_1)=g(x_2)\implies x_1=x_2$.\end{itemize}

Since $e\colon(Z,\gamma)\to(X,\alpha)$ is an equalizer of $f$ and $g$,  we may identify $Z$ with the subset $$\{x\in X\mid f(x)=g(x)\}$$ of $X$ and identify $\gamma$ with the   restriction of $\alpha$ on $Z$; that is, $\gamma=\alpha|Z$.

Suppose that $d\colon \mathscr{P}(X,\alpha)\to (W,\lambda)$ is a Goguen map such that $d\circ \mathscr{P}f= d\circ \mathscr{P}g$. We need to show that there is a unique Goguen map $$\overline{d}\colon \mathscr{P}(Z,\gamma)\to(W,\lambda)$$ satisfying $d= \overline{d}\circ \mathscr{P}e$. Uniqueness is obvious since $\mathscr{P}e$ is an epimorphism.

Before proving the existence of $\overline{d}$, we show that for all $\xi_1,\xi_2\in\sQ^X$, \[\xi_1|Z=\xi_2|Z\implies d(\xi_1)=d(\xi_2).\]  To see this, define $\xi\in\sQ^Y$ by \[\xi(y)=\begin{cases}
\xi_1(x) &   y=f(x) ~ \text{for some}~ x\in X, \\ \xi_2(x) &  y=g(x)~ \text{for some}~ x\in X,\\ 1 & {\rm otherwise}.\end{cases}\] That $\xi$ is well-defined follows from that $\xi_1|Z=\xi_2|Z$ and the aforementioned facts about $f$ and $g$.
Since $\mathscr{P}f(\xi)=\xi_1$ and $\mathscr{P}g(\xi)=\xi_2$, it follows that $d(\xi_1)=d(\xi_2)$, as desired.

For each $\zeta\in \sQ^Z$, define $E(\zeta)\in\sQ^X$ by \[E(\zeta)(x)=\begin{cases}\zeta(x) & x\in Z,\\ 1 & x\notin Z.\end{cases}\]
 Then $E\colon \mathscr{P}(Z,\gamma)\to\mathscr{P}(X,\alpha)$ is a Goguen map, because for each $\zeta\in\sQ^Z$,  \[\gamma^\uparrow(\zeta)= \bw_{x\in Z}\zeta(x)\ldd\gamma(x)= \bw_{x\in X}E(\zeta)(x)\ldd\alpha(x) = \alpha^\uparrow(E(\zeta)).\]

Let $\overline{d}=d\circ E$. We claim that $\overline{d}$ satisfies the requirement. For each $\xi\in\sQ^X$, since  the restrictions of  $E\circ\mathscr{P}e(\xi)$  and $\xi$ on $Z$ are equal, i.e., $(E\circ\mathscr{P}e(\xi)) |Z=\xi|Z,$  it follows that \[\overline{d}\circ\mathscr{P}e(\xi)=d(E\circ\mathscr{P}e(\xi))= d(\xi),\]  which completes the proof.
\end{proof}

%In order to describe the Kleisli category of the monad $\mathfrak{P}$, we need a lemma.

%\begin{lem} Let $(X,\alpha)$ and $(Y,\beta)$ be objects of $\setQ$ and let $f\colon X\to \sQ^{\sQ^Y}$ be a map. Then $$f\colon(X,\alpha)\to \mathfrak{P}(Y,\beta)$$ satisfies the Goguen condition if and only if  so does its transpose $$\overline{f}\colon(\sQ^Y,\beta_\dag^\uparrow)\to (\sQ^X,\alpha_\dag^\uparrow), \quad \overline{f}(\lambda)(x)=f(x)(\lambda).$$ \end{lem}  \begin{proof} Similar to that of Proposition \ref{adjunction P Pdagger}.\end{proof}

%This follows from that \begin{align*}&\quad\quad\quad \text{$f\colon(X,\alpha)\to \mathfrak{P}(Y,\beta)$ satisfies the Goguen condition}\\ &\iff \alpha(x)\leq (\beta^\uparrow_\dag)^\uparrow(f(x))~\text{for all $x\in X$}\\ &\iff  \alpha(x)\leq f(x)(\lambda)\ldd(\beta\searrow\lambda)~\text{for all $x\in X$ and $\lambda\in\sQ^Y$}\\ &\iff \beta\searrow\lambda\leq \alpha(x)\searrow \overline{f}(\lambda)(x)~\text{for all $x\in X$ and $\lambda\in\sQ^Y$}\\ &\iff \overline{f}\colon(\sQ^Y,\beta_\dag^\uparrow)\to (\sQ^X,\alpha_\dag^\uparrow)~\text{satisfies the Goguen condition}.\end{align*}

%With help of the above lemma, one sees that the Kleisli category   of the monad $\mathfrak{P}$ is isomorphic to the category for which: \begin{itemize}\item    objects are fuzzy sets $(X,\alpha)$, $(Y,\beta)$, $\cdots$; \item  a morphism  $h$ from $(X,\alpha)$ to $(Y,\beta)$ is a Goguen map $h\colon(\sQ^Y,\beta_\dag^\uparrow)\to (\sQ^X,\alpha_\dag^\uparrow)$; \item  composition is the usual composition of maps.\end{itemize}

Next, we show that  for a commutative quantale, the (covariant) powerset monad $(\mathscr{U},\sfm,\sfe)$ is a submonad of the double contravariant powerset monad  $(\mathfrak{P},\mu,\eta)$. %That means, there is a natural transformation $\kappa\colon\mathscr{U}\to\mathfrak{P}$ such that $\kappa$ is a monomorphism (in the category of functors) and  the diagrams \[\bfig\qtriangle[\id`\mathscr{U}`\mathfrak{P};\sfe`\eta`\kappa] \square(1100,0)[\mathscr{U}^2`\mathfrak{P}^2` \mathscr{U}`\mathfrak{P};\kappa*\kappa`\sfm` \mu`\kappa]\efig\] are commutative, where $\kappa*\kappa$ stands for the horizontal composite of $\kappa$ with itself.

\begin{lem}\label{P to P-2}
For each set $X$, the map $$j_X\colon \sQ^X\to\sQ^{\sQ^X},\quad j_X(\lambda)(\gamma)=\gamma\swarrow\lambda$$ is injective. The assignment $X\mapsto j_X$ defines a natural transformation from $\exp_\sQ$ to $\exp_\sQ^{-2}$.
\end{lem}
\begin{proof}That $j_X$ is injective is clear. It remains to check that $\{j_X\}_X$ is a natural transformation; that is, for each map $f\colon X\to Y$, the square \[\bfig\square<575,525>[\exp_\sQ X`\exp_\sQ^{-2} X` \exp_\sQ Y`\exp_\sQ^{-2} Y;j_X`\exp_\sQ f` \exp_\sQ^{-2} f`j_Y]\efig\] is commutative. This is easy since for all $\lambda\in\sQ^X$ and $\gamma\in\sQ^Y$, by Lemma \ref{image vs preimage} and equation (\ref{P-2f}) we have   \begin{align*}j_Y\circ\exp_\sQ f(\lambda)(\gamma)&=\gamma\swarrow f(\lambda) =\gamma\circ f\swarrow\lambda   =  \exp_\sQ^{-2}f(j_X(\lambda))(\gamma). \qedhere \end{align*} \end{proof}

If $\sQ$ is commutative, then for all set $X$ and all $\lambda,\gamma\in\sQ^X$,  \[\lambda\searrow\gamma=\bw_{x\in X}(\lambda(x)\ra\gamma(x))= \gamma\swarrow\lambda.\]  In this case  we  write $$\sub_X(\lambda,\gamma)\coloneqq\lambda\searrow\gamma=\gamma\swarrow\lambda = j_X(\lambda)(\gamma).$$

\begin{thm}\label{P is submonad} Let $\sQ$ be a commutative quantale. Then the powerset monad $(\mathscr{U},\sfm,\sfe)$ is a submonad of the double contravariant powerset monad $(\mathfrak{P},\mu,\eta)$. \end{thm} %\begin{enumerate}[label={\rm(\roman*)}] \setlength{\itemsep}{0pt} \item For each $(X,\alpha)$ of $\setQ$ and each $\lambda\in\sQ^X$, \(\alpha^\downarrow(\lambda)=  (\alpha_\dag^{\uparrow})^\uparrow(j_X(\lambda)).\) \item The powerset monad $(\mathscr{U},\sfm,\sfe)$ is a submonad of the double contravariant powerset monad $(\mathfrak{P},\mu,\eta)$.\end{enumerate}

\begin{proof} We prove the conclusion in two steps.

\textbf{Step 1}.  For each  object $(X,\alpha)$ of $\setQ$ and each $\lambda\in\sQ^X$, \(\alpha^\downarrow(\lambda)=  (\alpha_\dag^{\uparrow})^\uparrow(j_X(\lambda)).\)

Since $$\alpha^\downarrow(\lambda)=\sub_X(\lambda,\alpha)= j_X(\lambda)(\alpha) $$  and $$\sub_X(\gamma_1,\gamma_2)\leq j_X(\lambda)(\gamma_1)\ra j_X(\lambda)(\gamma_2)$$ for all $\gamma_1,\gamma_2\in\sQ^X$, it suffices to show that  $$ (\alpha_\dag^{\uparrow})^\uparrow(\Lambda)=\Lambda(\alpha)  $$ whenever $\Lambda\colon\sQ^X\to\sQ$  satisfies   $$\sub_X(\gamma_1,\gamma_2)\leq \Lambda(\gamma_1)\ra\Lambda(\gamma_2). $$ %for all $\gamma_1,\gamma_2\in\sQ^X$.

Since $\sQ$ is commutative, by definition we have \begin{align*}(\alpha_\dag^{\uparrow})^\uparrow(\Lambda) &= \bw_{\gamma\in\sQ^X}\sub_X(\alpha,\gamma)\ra\Lambda(\gamma).  \end{align*} Since $\sub_X(\alpha,\alpha) \geq k$, then \[(\alpha_\dag^{\uparrow})^\uparrow(\Lambda)\leq k\ra \Lambda(\alpha)  = \Lambda(\alpha).\] Conversely, since \[\sub_X(\alpha,\gamma)\ra\Lambda(\gamma) \geq (\Lambda(\alpha)\ra\Lambda(\gamma)) \ra\Lambda(\gamma) \geq \Lambda(\alpha)\] for all $\gamma\in\sQ^X$,   then  \begin{align*}(\alpha_\dag^{\uparrow})^\uparrow(\Lambda) &= \bw_{\gamma\in\sQ^X}\sub_X(\alpha,\gamma)\ra\Lambda(\gamma)\geq \Lambda(\alpha).  \end{align*}

\textbf{Step 2}. $(\mathscr{U},\sfm,\sfe)$ is a submonad of  $(\mathfrak{P},\mu,\eta)$.

By  \textbf{Step 1} one sees that for each $(X,\alpha)$ of $\setQ$, the map \[\kappa_{(X,\alpha)}\colon \mathscr{U}(X,\alpha)\to \mathfrak{P}(X,\alpha), \quad \lambda\mapsto \sub_X(\lambda,-)\] satisfies the Goguen condition, hence $\kappa=\{\kappa_{(X,\alpha)}\}$ is a natural transformation from $\mathscr{U}$ to  $\mathfrak{P}$, and it is a lifting of the natural transformation $j=\{j_X\}$ in the above lemma.

It is clear that, as a morphism between functors,    $\kappa$ is a monomorphism and  $\eta=\kappa\circ\sfe$. So, to see that $(\mathscr{U},\sfm,\sfe)$ is a submonad of $(\mathfrak{P},\mu,\eta)$, we only need to show that the square \[\bfig  \square[\mathscr{U}^2`\mathfrak{P}^2` \mathscr{U}`\mathfrak{P};\kappa*\kappa`\sfm` \mu`\kappa]\efig\] is commutative. %, where $\kappa*\kappa$ stands for the horizontal composite of $\kappa$ with itself.
Since $\kappa$  is a lifting of  $j$, it suffices to show that for each set $X$, the following square  is commutative: \[\bfig\square<650,550>[\exp_\sQ^2X`\exp_\sQ^{-4}X` \exp_\sQ X`\exp_\sQ^{-2} X;(j*j)_X`\sfm_X` \mu_X`j_X]\efig\] For this we calculate:   for all $\Lambda\in\sQ^{\sQ^X}$ and   $\lambda\in\sQ^X$,
\begin{align*}\mu_X\circ(j*j)_X(\Lambda)(\lambda)&= \mu_X\circ j_{\sQ^{\sQ^X}}\circ \exp_\sQ j_X (\Lambda)(\lambda) % & (\text{Definition of $j*j$})
\\ &=j_{\sQ^{\sQ^X}}(j_X (\Lambda))(\widehat{\lambda}) % & (\text{Equation (\ref{def of mu}}))
\\ &= \bw_{\Xi\in \sQ^{\sQ^X}} j_X (\Lambda)(\Xi)\ra \Xi(\lambda) %& (\text{Definition of $j_{\sQ^{\sQ^X}}$})
\\ &= \bw_{\gamma\in\sQ^X}\Lambda(\gamma)\ra \sub_X(\gamma,\lambda) %& \Big(j_X (\Lambda)(\Xi)=\bv_{j_X(\gamma)=\Xi} \Lambda(\gamma)\Big)
\\ &= \sub_X\Big(\bv_{\gamma\in\sQ^X}\Lambda(\gamma)\with\gamma, \lambda\Big) \\ &= j_X\circ\sfm_X(\Lambda)(\lambda). \qedhere \end{align*}  \end{proof}
 
The fact that the monad $(\mathfrak{P},\mu,\eta)$ is a lifting of  $(\exp_\sQ^{-2},\mu,\eta)$ is useful. As an application, we describe here a submonad of $(\mathfrak{P},\mu,\eta)$  by lifting the $\sQ$-filter monad $\FQ$ in the category of sets, the latter is a submonad of  $(\exp_\sQ^{-2},\mu,\eta)$. The same idea can be used to construct some other  monads in $\setQ$.

The following definition is a slight modification of that of  $\sQ$-filter  in \cite{EG92,Hoehle2001,LZZ2021}.

\begin{defn} \label{Q-filter} A $\sQ$-filter   on a  set $X$ is a  map $F\colon \sQ^X\to \sQ$ subject to the following conditions:  for all    $\lam,\gamma\in \sQ^X$, \begin{enumerate}[label=(F\arabic*)] \setlength{\itemsep}{0pt}
\item \label{FF1} $F(k_X)\geq k$, where $k_X$ is the constant map $X\to\sQ$ with value $k$; \item \label{FF2} $F(\lam)\wedge F(\gamma)\leq F(\lam\wedge\gamma)$; \item \label{FF3} $\gamma\swarrow\lambda  \leq F(\gamma)\ldd F(\lam)$.
 \item\label{FF4} $F(r_X)\leq r$  for all $r\in\sQ$, where $r_X$ is the constant map $X\to\sQ$ with value $r$.   \end{enumerate}
\end{defn}

In  presence of \ref{FF3}, the inequalities in \ref{FF2} and \ref{FF4} are actually   equalities.

For each set $X$, write $$\FQ(X)$$  for the set of $\sQ$-filters on $X$. For each $f\colon X\to Y$ and each $F\in\FQ(X)$, define \[f(F)\colon\sQ^Y\to\sQ\]   by \[f(F)(\gamma)=F(\gamma\circ f).\] Then $f(F)$ is a $\sQ$-filter on $Y$. In this way we obtain a functor \[\FQ\colon{\sf Set}\to{\sf Set}.\]

The $\sQ$-filter functor $\FQ$ is a subfunctor of   $\exp_\sQ^{-2}\colon{\sf Set}\to{\sf Set}$,   indeed, it can be made into a submonad of   $(\exp_\sQ^{-2}, \mu,\eta)$, as we see below.

\begin{lem}\label{closed under multiplication}
For each $\sQ$-filter $\mathbb{F}$ on   $\FQ(X)$, the map \[\sigma(\mathbb{F})\colon\sQ^X\to\sQ, \quad \sigma(\mathbb{F})(\lam)=\mathbb{F}(\widehat{\lam})  \] is a $\sQ$-filter on $X$, where $\widehat{\lam}\colon\FQ(X)\to\sQ$ is   given by \(\widehat{\lam}(F)= F(\lam). \)   \end{lem}

\begin{proof} That $\sigma(\mathbb{F})$ satisfies \ref{FF1}, \ref{FF2}  and \ref{FF4} is clear,
it remains to check that it satisfies \ref{FF3}. We calculate: for all $\lam,\gamma\in\sQ^X$, \begin{align*}\gamma\swarrow\lam &\leq\bw_{F\in\FQ(X)} F(\gamma)\ldd F(\lam) \\ &=\widehat{\gamma}\swarrow\widehat{\lam} \\ &\leq \mathbb{F}(\widehat{\gamma})\ldd\mathbb{F}(\widehat{\lam}) %& (\mathbb{F}~\text{is a $\sQ$-filter})
\\ &= \sigma(\mathbb{F})(\gamma)\ldd\sigma(\mathbb{F})(\lam), \end{align*} which completes the proof.
\end{proof}
The $\sQ$-filter $\sigma(\mathbb{F})$ is called the \emph{diagonal $\sQ$-filter}, or the \emph{Kowalsky sum},  of $\mathbb{F}$.
The diagonal $\sQ$-filter is closely related to the multiplication of the monad  $(\exp_\sQ^{-2},\mu,\eta)$. Let $i$ be the inclusion transformation of the   functor $\FQ$ in $\exp_\sQ^{-2}$. Then for   each $\mathbb{F}\in \FQ^2(X)$,  \[\sigma(\mathbb{F})= \mu_X\circ(i*i)_X(\mathbb{F}),  \] where  $i*i$ stands for the horizontal composite of $i$ with itself. This shows that the functor $\FQ$ is closed under the multiplication  $\mu$, hence     $\mu$  induces a natural transformation from $\FQ^2$ to $\FQ$, which is also denoted by $\mu$.

For each $x$ of $X$,   \[\eta_X(x)\colon\sQ^X\to\sQ,  \quad \eta_X(x)(\lam)=\lam(x)\]  is a $\sQ$-filter, hence the unit  of the monad $(\exp_\sQ^{-2},\mu,\eta)$ factors through  $\FQ$. This means that $\eta$ can be viewed as a natural transformation  from the identity functor to $\FQ$.

Since $\eta$ factors through $\FQ$ and $\FQ$ is closed under the multiplication $\mu$,   the triple \[(\FQ,\mu,\eta)\] is a monad in the category of sets, a submonad  of  $(\exp_\sQ^{-2},\mu,\eta)$.

Now we lift the monad $(\FQ,\mu,\eta)$ to a monad in $\setQ$. For each object $(X,\alpha)$ of $\setQ$, let \[\mathfrak{F}(X,\alpha)= (\FQ(X),(\alpha_\dag^\uparrow)^\uparrow),\] where for each $\sQ$-filter $F$ on $X$, \[(\alpha_\dag^\uparrow)^\uparrow(F)= \bw_{\gamma\in\sQ^X}\Big(F(\gamma)\ldd\bw_{x\in X}\alpha(x)\rdd\gamma(x)\Big).\]
Then we obtain a functor $$\mathfrak{F}\colon\setQ\to\setQ,$$  which is  a subfunctor of the  functor $\mathfrak{P}$.

\begin{prop} The triple $(\mathfrak{F},\mu,\eta)$ is a submonad of the monad $(\mathfrak{P},\mu,\eta)$ in $\setQ$, and it is a lifting of the $\sQ$-filter monad $(\FQ,\mu,\eta)$. \end{prop}

%The monad $(\mathfrak{F},\mu,\eta)$ in   $\setQ$  is an analog of  the filter monad in the category of sets.

Besides the covariant $\sQ$-powerset monad, the $\sQ$-filter monad, and the double contravariant $\sQ$-powerset monad, some other monads in $\Set$  can also be lifted to $\setQ$. For instance, the   powerset monad in Example \ref{powerset monad}  and the  list monad  (see e.g. \cite[page 156]{Riehl}). For each object $(X,\alpha)$ of $\setQ$ and each subset $A\subseteq X$, let $$\alpha_P(A)=\bw_{a\in A}\alpha(a).$$  Then the assignment $(X,\alpha)\mapsto(2^X,\alpha_P)$ gives rise to a functor on $\setQ$, which leads to a lifting of the  powerset monad to $\setQ$. The  list monad  in $\Set$  can be lifted to $\setQ$ in a similar way.

It should be noted that there exist monads in $\setQ$ that are not lifting of any monad in the category of sets, the monad $\mathbb{P}^2$   constructed in Demirci \cite[Section 4]{Demirci21} provides such an example.

\end{document}